\documentclass[12pt,a4paper]{amsart}

\usepackage{a4wide}

\usepackage{amsmath}
\usepackage{amssymb}
\usepackage{amsthm}
\usepackage[matrix,arrow]{xy}

\newtheorem{Theorem}{Theorem}[section]

\newtheorem{Corollary}[Theorem]{Corollary}
\newtheorem{Proposition}[Theorem]{Proposition}

\newtheorem{Remark}[Theorem]{Remark}

\def \dim{{\mbox {dim}}\,}

\def\V{\mbox{Var}}

\def\Z{{\mathbb Z}}
\def\R\re
\def\V{\bf V}

\def \vr{\varphi}
\def \R{{\mathbb R}}
\def \N{{\mathbb N}}
\def \Q{{\mathbb Q}}

\def \T{{\mathbb T}}

\def \V{{\bf V}}

\def \ga{\gamma}

\def\P{{\mathcal P}}
\def\CZ{{\text{CZ}}}

\begin{document}
\title[Equivariant symplectic homology of Anosov contact structures]{Equivariant symplectic homology\\ of Anosov contact structures }

\author[L. Macarini]{Leonardo Macarini}
\address{Universidade Federal do Rio de Janeiro, Instituto de Matem\'atica,
Cidade Universit\'aria, CEP 21941-909 - Rio de Janeiro - Brazil}
\email{leonardo@impa.br}

\author[G.P. Paternain]{Gabriel P. Paternain}
 \address{ Department of Pure Mathematics and Mathematical Statistics,
University of Cambridge,
Cambridge CB3 0WB, UK}
 \email {g.p.paternain@dpmms.cam.ac.uk}


\begin{abstract}
We show that the differential in positive equivariant symplectic homology or linearized contact homology vanishes for non-degenerate Reeb flows with a continuous invariant Lagrangian subbundle (e.g. Anosov Reeb flows). Several applications are given, including obstructions to the existence of these flows and abundance of periodic orbits for contact forms representing an Anosov contact structure.

\end{abstract}

\maketitle

\section{Introduction}

A flow $\vr_t$ on a closed manifold $M$ is said to be {\it Anosov} if there is a Riemannian metric on $M$, a constant $0<\lambda<1$ and subbundles $E^s$ and $E^u$ of $TM$ such that $TM = E^s \oplus E^u \oplus \text{span}\{X\}$, where $X = \dot\vr_t$, $d\vr_t(E^{s,u}) = E^{s,u}$ for every $t \in \R$ and
$$ \|d\vr_t(x)|_{E^s}\| \leq \lambda^{t}, \hskip 1cm \|d\vr_{-t}(x)|_{E^u}\| \leq \lambda^{t} $$
for every $x \in M$ and $t>0$. It is well known that $E^{s,u}$ are continuous subbundles.
Geodesic flows of metrics of negative sectional curvature are examples of Anosov flows as well as suspensions of Anosov diffeomorphisms. There are several other classes obtained by surgery, but in general
it is an open problem to classify those manifolds which support an Anosov flow, even in
dimension three. As far as we are aware, it is not even known if there
exists an Anosov flow on a simply connected closed manifold.

In this paper we are interested in a special class of Anosov flows, namely those which are also Reeb flows of contact forms on $M$. Examples of these are again geodesic flows of negatively curved manifolds (Riemannian or Finsler), but there are also more exotic examples (in dimension three) obtained by surgery which are not topologically orbit equivalent to algebraic flows, see \cite{FH}.
We shall call a contact structure $\xi$ Anosov if it admits a contact form whose Reeb flow is Anosov. 

It is well known and easy to see from the definition above that the subbundles
$E^{s,u}$ must be contained in $\xi$ and are Lagrangian; this already implies that all odd Chern classes
of $\xi$ satisfy $2c_{2i+1}(\xi)=0$ and $c_{1}(\xi)=0$ if $E^{s,u}$ is orientable. For geodesic flows
a theorem of R. Ma\~n\'e \cite{Ma} asserts that the existence of a continuous invariant Lagrangian subbundle implies the absence of conjugate points and hence every closed geodesic has
index zero. Ma\~n\'e's theorem is inspired by an earlier theorem of
W. Klingenberg \cite{K} (see also \cite{A2}) which states that a manifold
whose geodesic flow is Anosov must be free of conjugate points. The results
in this note should be seen as extensions of these results for geodesic flows to the general setting of Reeb flows.

\section{Results}
Before stating the results we describe precisely our setting. Let $(W^{2n},\omega)$, $n \geq 2$, be a symplectic manifold with contact type boundary, that is, $W$ is a compact symplectic manifold whose boundary has a transverse vector field $X$ pointing outwards and satisfying ${\mathcal L}_X\omega=\omega$. Suppose that $(W^{2n},\omega)$ satisfies $\int_{\mathbb T^2}f^*\omega=0$ for every smooth map $f: \mathbb T^2 \to W$ defined on the 2-torus $\mathbb T^2$. Denote by $M$ the boundary of $W$ and let $\alpha$ be its contact form. We shall assume throughout that the map $\pi_1(M) \rightarrow \pi_1(W)$ induced by the inclusion is injective.  For the sake of simplicity we shall also assume that the first Chern class $c_1(TW)$ vanishes, although this hypothesis most likely can be removed using Novikov rings. Let $a$ be a free homotopy class of loops in $W$.

Now, suppose that  $\alpha$ is non-degenerate. This means that the linearized first return map of every periodic orbit of the Reeb flow of $\alpha$ does not have one as an eigenvalue. Following the work of C. Viterbo \cite{Vit2}, F. Bourgeois and A. Oancea introduced in \cite{BO1} the $S^1$-equivariant symplectic homology $SH_*^{S^1,a}(W)$ as well as its positive part $SH_*^{S^1,+}(W)$ when $a=0$. It is defined as the direct limit of $S^1$-equivariant Floer homologies of admissible Hamiltonians and is an invariant of the contact structure $\xi$ in the following sense. Given another non-degenerate contact form $f\alpha$ on $M$, where $f: M \to \R$ is a positive function, then we can naturally construct from $W$ (using the symplectic completion of $(W,c\omega)$, where $c$ is a suitably chosen positive constant) a new symplectic manifold with contact type boundary $W_f$ such that $\partial W_f$ can be identified with $M$, the contact form on $\partial W_f$ is given by $f\alpha$ and the  $S^1$-equivariant symplectic homologies of $W_f$ and $W$ are isomorphic. In what follows we will use the same notation $SH_*^{S^1,a,+}(W)$ to denote $SH_*^{S^1,a}(W)$ if $a \neq 0$ and $SH_*^{S^1,+}(W)$ if $a=0$.

The Morse-Bott spectral sequence converges to $SH_*^{S^1,a,+}(W)$ and its second page is given by
\begin{equation}
\label{eq:2nd_page}
E_{j,k}^2 = \bigoplus_{\gamma \in \P^a_k(\alpha)} \Q,
\end{equation}
where $\P^a_k(\alpha)$ is the set of good periodic orbits of the Reeb flow of $\alpha$ with free homotopy class $a$ and Conley-Zehnder index equal to $k$, see \cite{FSK}. Recall that a periodic orbit is said to be {\it good} if it is not an even iterate of a simple periodic orbit whose iterates have indices with different parities. Otherwise, it is called bad. In order to define the Conley-Zehnder index, we have to fix a representative $l_a$ of $a$ and a symplectic trivialization of $TW$ along $l_a$. However, the parity of the Conley-Zehnder index does not depend on these choices and hence such an extra structure is immaterial for the purposes of this work, although it will be implicitly assumed.

The differential of equivariant symplectic homology counts rigid equivariant Floer trajectories connecting good periodic orbits. The precise definition of the differential is not essential here and hence we will not go into details about it. We only need to mention that, although $E_{j,k}^2$ involves only the dynamics of the Reeb flow and consequently can be computed in many examples, the differential rests on solutions of a partial differential equation and its computation is, in general, a very difficult task.

A continuous Lagrangian subbundle on $M$ is a continuous isotropic subbundle of maximal dimension in $\xi$. We say that a Lagrangian subbundle is invariant by the Reeb flow if it is invariant by its linearized flow. The main observation of this note states that the existence of a continuous Lagrangian subbundle invariant under the Reeb flow implies the vanishing of the differential. This can be seen as a hyperbolic version of the result in \cite{AM} asserting that the differential in contact homology vanishes for suitable contact forms on good toric contact manifolds.

\begin{Theorem}
\label{mainthm}
If the Reeb flow of $M$ has a continuous invariant Lagrangian subbundle $E$ then every periodic orbit in the same free homotopy class has index with the same parity. In particular,
$$ SH_k^{S^1,a,+}(W) = \bigoplus_{\gamma \in \P^a_k(\alpha)} \Q. $$
Moreover, every contractible closed orbit has even index and if $E$ is orientable then every periodic orbit has even index.
\end{Theorem}

\begin{Remark}
{\rm Under the isomorphism between equivariant symplectic homology and linearized contact homology established in \cite{BO2}, the theorem above implies the vanishing of the differential in linearized contact homology whenever the Reeb flow has a continuous invariant Lagrangian subbundle.}
\end{Remark}

\begin{Remark}
{\rm A continuous Lagrangian subbundle $E$ is orientable if and only its first Stiefel-Whitney class $w_1(E) \in H^1(M,\Z_2)$ vanishes. In particular, this condition is automatically satisfied if M is simply connected or, more generally, if $H^1(M,\Z_2)=0$.}
\end{Remark}

Now, let us state some applications of Theorem \ref{mainthm}. The first ones give obstructions to the existence of invariant Lagrangian subbundles for Reeb flows and will be proved in Section \ref{proof:obstructions}. Define $SH_*^{S^1,+}(W) = \oplus_a SH_*^{S^1,a,+}(W)$, where the sum runs over all the free homotopy classes $a$ of $W$.

\begin{Theorem}
\label{obstruction1}
The Reeb flow of a non-degenerate tight contact form on $S^3$ admits no continuous invariant Lagrangian subbundle.
\end{Theorem}

\begin{Theorem}
\label{obstruction2}
Suppose that $SH_*^{S^1,+}(W)$ is uniformly bounded, that is, there exists a constant $C>0$ such that $\mbox{\rm dim}\, SH_k^{S^1,+}(W) < C$ for every $k$. Then there is no non-degenerate contact form representing $\xi$ whose Reeb flow has an orientable continuous invariant Lagrangian subbundle and infinitely many simple hyperbolic periodic orbits.
\end{Theorem}

It is well known that Anosov flows have infinitely many simple periodic orbits, see, for instance, \cite{HK}. Therefore,  we have the following corollary of Theorem \ref{obstruction2}.

\begin{Corollary}
Suppose that $H^1(M,\Z_2)=0$ and that $SH_*^{S^1,+}(W)$ is uniformly bounded. Then $\xi$ is not Anosov.
\end{Corollary}

Given a contact manifold $M$ and an exact symplectic manifold $(V,\lambda)$, an exact contact embedding $f: M \to V$ is an embedding such that $f(M) \subset V$ is bounding and there exists a contact form $\alpha$ on $M$ such that $\alpha - f^*\lambda$ is exact. Following \cite{FSK}, we say that $M$ is index-positive if every periodic orbit of the Reeb flow has positive mean index. Under this assumption, it is proved in \cite{FSK} that if $M$ has a displaceable exact contact embedding into an exact  convex symplectic manifold $V$ then
$$ SH_*^{S^1,+}(W) \simeq H_{*+n-1}(W,M), $$
where $W$ denotes the compact component of $V\setminus M$.

\begin{Corollary}
If a simply connected contact manifold is index positive and admits a displaceable exact contact embedding into an exact  convex symplectic manifold then its contact structure is not Anosov. In particular, the canonical contact form of $S^{2n-1}$  is not Anosov for every $n\geq 2$.
\label{cor:esfera}
\end{Corollary}

\begin{Remark}
{\rm As mentioned in \cite{FSK}, it is conceivable that the equivariant symplectic homology of a contact manifold admitting a displaceable exact contact embedding vanishes (the corresponding result for non-equivariant symplectic homology was proved in \cite[Theorem 97]{Rit}). This would enable us to remove the hypothesis of index positivity above.}
\end{Remark}

There is a natural filtration in the complex of positive equivariant symplectic homology given by the action. Given $T \in \R$ we denote the truncated homology by $SH_*^{S^1,+,T}(W)$. Following \cite{Se, McL}, we define the growth rate of $SH_*^{S^1,+}(W)$ as
$$ \Gamma(W) = \limsup_{T\to\infty}\frac{1}{\log T}\log \dim \iota(SH_*^{S^1,+,T}(W)), $$
where $\iota: SH_*^{S^1,+,T}(W) \to SH_*^{S^1,+}(W)$ is the map induced by the inclusion. The argument in \cite[Section 4a]{Se} shows that $\Gamma(W)$ is an invariant of the contact structure. Since this argument will be important in the proof of Theorem \ref{infper} below, we will reproduce it in Section \ref{proof:infpers}.

A classical result due to Bowen \cite{Bow} establishes that if $\vr_t$ is an Anosov flow then
\begin{equation}
\label{bowen}
h_{top}(\vr_t) = \lim_{T\to\infty}\frac{1}{T} \log P_T(\vr_t) > 0,
\end{equation}
where $h_{top}(\vr)$ is the topological entropy of $\vr_t$ and $P_T(\vr_t)$ stands for the number of periodic orbits of $\vr_t$ with period less than or equal to $T$. Note that this period does not have to be the minimal one.

Given a non-degenerate contact form $\alpha$ denote by $P_T(\alpha)$ (resp. $P^g_T(\alpha)$) the number of periodic orbits (resp. {\it good} periodic orbits) of $\alpha$ with period less than or equal to $T$. It is easy to see that $P_T(\alpha) \leq 2P^g_T(\alpha) \leq 2P_T(\alpha)$. Indeed, given a simple periodic orbit $\gamma$, the number of bad iterates of $\gamma$ with action less than or equal to $T$ is at most half of the total number of iterates with action less than or equal to $T$. Therefore,
\begin{equation}
\label{grate_good}
\lim_{T\to\infty} \frac{1}{T} \log P_T(\alpha) = \lim_{T\to\infty} \frac{1}{T} \log P^g_T(\alpha).
\end{equation}
It follows from Bowen's result, equation \eqref{grate_good} and Theorem \ref{mainthm} that if $\xi$ is Anosov then $\Gamma(W) = \infty$. This has several consequences as the next three results show. Recall that a hypersurface in the cotangent bundle of a closed manifold is fiberwise starshaped if its intersection with each fiber is non-empty and starshaped.

\begin{Theorem}
\label{obstruction3}
If $\Gamma(W) < \infty$ then $\xi$ is not Anosov. In particular, there is no fiberwise starshaped hypersurface in $T^*\mathbb T^n$, $n\geq 2$, whose Reeb flow is Anosov. \label{cor:toros}
\end{Theorem}

\begin{Remark}
\label{dimension3}
{\rm Corollary \ref{cor:esfera} and Theorem \ref{cor:toros}
are only interesting when the contact manifold has dimension $\geq 5$.
This is because it is known that if a 3-manifold $M$ supports an
Anosov flow, then $\pi_{1}(M)$ must grow exponentially \cite{PT}.
In fact, it is also known that an Anosov flow on a 3-manifold cannot have
closed contractible orbits \cite{A}.
In particular an Anosov contact structure on a 3-manifold must be tight
due to Hofer's results in \cite{Ho,HoK}.}
\end{Remark}

\begin{Remark}{\rm A potential application of Theorem \ref{cor:toros} was pointed out to us
by Ivan Smith. In \cite{McL}, M. McLean shows that if the boundary of a Liouville domain is algebraically Stein fillable,
then the growth rate of symplectic homology is finite. If one could show that this finiteness result
implies the analogous result $\Gamma(W)<\infty$ for {\it equivariant} symplectic homology, then
it would follow that if the contact structure of a Liouville domain is Anosov, then it is not algebraically Stein fillable.
This is precisely the situation of unit contangent bundles of manifolds of negative curvature.
In other words, hyperbolic dynamics of the Reeb flow obstructs affine fillings.

}\label{rem:ivan}
\end{Remark}

It follows from \cite[Theorem 1.4]{HM} and the isomorphism between positive equivariant symplectic homology and linearized contact homology established in \cite{BO2} that if $\Gamma(W) > 1$, then every contact form $\alpha$ representing $\xi$ has infinitely many simple periodic orbits. In particular, every contact form (possibly degenerate) representing an Anosov contact structure has infinitely many simple periodic orbits.
This result can be strengthened under the assumption that $\alpha$ is non-degenerate. Given a contact form $\alpha$ denote by $h_{top}(\alpha)$ the topological entropy of the Reeb flow of $\alpha$.

\begin{Theorem}
\label{infper}
Suppose that $\xi$ is Anosov and let $\alpha$ be a contact form representing $\xi$ whose Reeb flow is Anosov. Let $f\alpha$ be a non-degenerate contact form, where $f: M \to \R$ is a positive function. Then
$$  \liminf_{T\to\infty}\frac{1}{T} \log P_T(f\alpha) \geq \frac{h_{top}(\alpha)}{\max f} > 0. $$
Moreover, if $f\alpha$ admits a continuous invariant Lagrangian subbundle then
$$ \frac{h_{top}(\alpha)}{\min f} \geq \liminf_{T\to\infty}\frac{1}{T} \log P_T(f\alpha) \geq \frac{h_{top}(\alpha)}{\max f}. $$
\end{Theorem}

Our last application is about the uniqueness of periodic orbits in a given free homotopy class. It is well known that given a closed Riemannian manifold $N$ with an Anosov geodesic flow, there is a unique closed geodesic in each non-trivial homotopy class of $N$ and no contractible closed geodesic \cite{K}. Thus, we have the following direct corollary of Theorem \ref{mainthm}.

\begin{Corollary}
Let $N$ be a closed manifold that admits a metric with an Anosov geodesic flow. Then every fiberwise starshaped hypersurface in $T^*N$ whose Reeb flow is non-degenerate and admits an orientable continuous invariant Lagrangian subbundle has a unique closed orbit in a given non-trivial homotopy class of $N$ and no contractible closed orbit.
\end{Corollary}

\section{Proof of Theorem \ref{mainthm}}
\label{proof:mainthm}

Let $E$ be a continuous Lagrangian subbundle invariant by the Reeb flow. Fix the free homotopy class $a$ of the periodic orbits and denote by $\alpha$ the contact form on $M$. Let $\gamma: S^1 \to M$ be a non-degenerate periodic orbit of $\alpha$ with homotopy class $a$ and fix a trivialization of $\xi$ along $\gamma$. Here $S^1 = \R/T\Z$, where $T$ is the period of $\gamma$. The linearized Reeb flow with this trivialization gives us a symplectic path whose Conley-Zehnder index will be denoted by $\mu_{\CZ}(\gamma)$. Although this number depends on the choice of the trivialization, its parity does not. Denote by $E_\gamma$ the vector bundle over $S^1$ given by the pullback of $E$ by $\gamma$.

\begin{Proposition}
\label{parity}
The Conley-Zehnder index of $\gamma$ is even if and only if $E_\gamma$ is orientable.
\end{Proposition}

\begin{proof}
Let $P$ denote the linearized Poincar\'e map of $\ga$. We will make use of the following equality:
\begin{equation}
(-1)^{\mu_{CZ}(\ga)}=(-1)^{n-1}\mbox{\rm sign}\,\det (I-P),
\label{eq:CZ}
\end{equation}
see \cite[Section 1.2]{EGH}. Consider a symplectic basis $\{e_{1},\dots,e_{n-1},f_{1},\dots,f_{n-1}\}$ of $\xi(\gamma(0))=\xi(\gamma(T))$ such that $\{e_{1},\dots,e_{n-1}\}$
is a basis of the Lagrangian subspace $E(\gamma(0))=E(\gamma(T))$.
Since $E_\ga$ is invariant $P$ has the matrix form
$$ P=\left(\begin{array}{ll}
A&B\\
0&C\\
\end{array}\right), $$
where $C=(A^{t})^{-1}$ and $A^{-1}B$ is symmetric. Note that $E_\ga$ is orientable along $\ga$ if and only if
$\det(A)>0$. Now we compute
\begin{align*}
\det(I-P)&=\det(I-A)\det(I-C)\\
&=\det(I-A)\det(I-(A^{t})^{-1})\\
&=\det(I-A)\det((A^{t})^{-1})\det(A^{t}-I)\\
&=(\det(A))^{-1}\det(I-A)(-1)^{n-1}\det(I-A^{t})\\
&=(\det(A))^{-1}(\det(I-A))^{2}(-1)^{n-1}.
\end{align*}
Thus the proposition follows from \eqref{eq:CZ}.
\end{proof}

Now, observe that $E_\gamma$ is orientable if and only if its first Stiefel-Whitney class $w_1(E_\gamma) \in H^1(S^1,\Z_2)$ vanishes. But
$$ w_1(E_\gamma) = \gamma^*w_1(E). $$
 Consequently, we conclude that the vanishing of $w_1(E_\gamma)$ depends only on the homology class of $\gamma$. Here we use in an essential way that $E$
is globally defined and continuous.
In particular, if $\bar\gamma$ is freely homotopic to $\gamma$ then $w_1(E_{\bar\gamma})$ vanishes if and only if $w_1(E_\gamma)$ vanishes. It follows from the previous proposition that the two closed
orbits $\gamma$ and $\bar{\gamma}$ must have the same parity and Theorem \ref{mainthm} is proved.
\qed

\section{Proof of Theorems \ref{obstruction1}, \ref{obstruction2} and \ref{obstruction3}}
\label{proof:obstructions}

\subsection{Proof of Theorem \ref{obstruction1}}
Suppose, by contradiction, that there exists a non-degenerate contact form representing $\xi$ whose Reeb flow has a continuous invariant Lagrangian subbundle. 
For the ball $D^4 \subset \R^4$ we have that
\begin{equation}
\label{chsphere}
SH_*^{S^1,+}(D^4) \simeq \begin{cases}
\Q & \text{ if } *=2k+1, k \in \N \\
0 & \text{ otherwise}.
\end{cases}
\end{equation}
By examining the proof of Theorem \ref{mainthm} we see that $E$ must be orientable since $S^3$ is simply connected. Thus
any closed orbit must have even index, which contradicts \eqref{chsphere}.
\qed

One can improve the argument above to get a little more.
It is easy to see that the existence of an invariant Lagrangian subbundle in dimension three implies that every periodic orbit is hyperbolic whenever the contact form is non-degenerate. The index of a hyperbolic periodic orbit $\ga$ satisfies
\begin{equation}
\label{hypind}
\mu_{\CZ}(\gamma^j) = j\mu_{\CZ}(\gamma)
\end{equation}
for every $j \in \N$. Since the differential vanishes, every periodic orbit has index greater than or equal to three. As a matter of fact, if there was a periodic orbit of index less than three then, by equation \eqref{hypind}, we would have a simple, and hence good, periodic orbit of index less than three. This contradicts the fact that the differential vanishes. By \eqref{chsphere} we conclude that given any odd number $k \geq 3$ there is a periodic orbit with index $k$. In particular, given two distinct prime numbers $k_1, k_2 \geq 3$ there are closed orbits $\ga_1$ and $\ga_2$ of index $k_1$ and $k_2$ respectively. Therefore, by equation \eqref{hypind}, $\ga_1$ and $\ga_2$ must be simple and, by \eqref{hypind} and the vanishing of the differential, $\dim SH_{k_1k_2}^{S^1,+}(D^4) \geq 2$ (notice that $\ga^{k_2}_1$ and $\ga^{k_1}_2$ are good since $k_1$ and $k_2$ are odd), contradicting \eqref{chsphere}.
This shows that if a contact (fillable) 3-manifold has the same positive equivariant symplectic homology as $S^3$ tight, then it does not support a non-degenerate Reeb flow with a continuous invariant Lagrangian subbundle.

\subsection{Proof of Theorem \ref{obstruction2}}
Suppose, by contradiction, that there exists an integer $N>0$ such that $\dim SH_k^{S^1,+}(W) < N$ for every $k$ and a contact form representing $\xi$ whose Reeb flow has an orientable continuous invariant Lagrangian subbundle $E$ and infinitely many simple hyperbolic periodic orbits. Let $\gamma_1, ..., \gamma_{N+1}$ be simple distinct hyperbolic periodic orbits and define
$$ k = \text{lcm}\{\mu_{\CZ}(\ga_1),...,\mu_{\CZ}(\ga_{N+1})\}. $$
The hypothesis that $E$ is orientable ensures, by  Theorem \ref{mainthm}, that every periodic orbit is good. Since the differential vanishes, it follows from equation \eqref{hypind} that $SH_{k}^{S^1,+}(W)$ has at least $N+1$ generators, a contradiction.
\qed

\subsection{Proof of Theorem \ref{obstruction3}}
As already mentioned in the introduction, the first assertion follows from Theorem \ref{mainthm}, equation \eqref{grate_good} and Bowen's result. For the second one, notice that the map induced by the inclusion $\pi_1(S^*\T^n) \to \pi_1(T^*\T^n)$ is injective if and only if $n\geq 3$. But by Remark \ref{dimension3} we can restrict ourselves to this case. Recall that a contact form $\alpha$ is Morse-Bott if its action spectrum $\text{Spec}(\alpha)$ is discrete and if for every $T \in \text{Spec}(\alpha)$ we have that $N_T := \{p \in M; \vr_T(p)=p\}$ is a smooth closed submanifold such that $d\alpha|_{N_T}$ is locally constant and $T_pN_T = \ker(d\vr_T - Id)_p$, see \cite{Bo}. Here $\vr_t$ denotes the Reeb flow of $\alpha$. A connected component of $N_T$ is called a Morse-Bott component with period $T$.

Consider the flat metric on the torus $\T^n$ and denote by $\alpha$ the corresponding contact form on $S^*\T^n$. It is well known that $\alpha$ is Morse-Bott and that
$$ \limsup_{T\to\infty} \frac{1}{\log T} \log P^{MB}_T(\alpha) < \infty, $$
where $P^{MB}_T(\alpha)$ is the number of Morse-Bott components of $\alpha$ with period (which does not need to be the minimal one) less than $T$. The Morse-Bott components of $\alpha$ are diffeomorphic to $\T^n$ and the corresponding orbit spaces are smooth manifolds diffeomorphic to $\T^{n-1}$. Fix a Morse function $f$ on $\T^{n-1}$ and denote by $\text{Crit}(f)$ the set of critical points of $f$. Given $T>0$ it follows from Lemma 2.3 in \cite{Bo} that there is a contact form $\alpha_T$ arbitrarily close to $\alpha$ such that every periodic orbit of $\alpha_T$ with action less than $T$ is non-degenerate and these orbits (up to an obvious reparametrization) are in bijection with the set $\text{Crit}(f) \times C^{MB}_T(\alpha)$, where $C^{MB}_T(\alpha)$ is the set of Morse-Bott components of $\alpha$ with period less than $T$, whose cardinality is given by $P^{MB}_T(\alpha)$. Denote by $W_T$ the obvious symplectic filling of $(M,\alpha_T)$. We have then that,
\begin{align*}
\limsup_{T\to\infty} \frac{1}{\log T} \log\dim SH_*^{S^1,+,T}(W) & = \limsup_{T\to\infty}   \frac{1}{\log T} \log\dim SH_*^{S^1,+,T}(W_T) \\
& \leq \limsup_{T\to\infty}  \frac{1}{\log T} \log P_T(\alpha_T) \\
& = \limsup_{T\to\infty} \frac{1}{\log T} \log P^{MB}_T(\alpha) < \infty,
\end{align*}
finishing the proof.
\qed

\section{Proof of Theorem \ref{infper}}
\label{proof:infpers}

First we will reproduce the argument of Seidel in \cite[Section 4a]{Se} that shows the invariance of the growth rate for (non-equivariant) symplectic cohomology under Liouville isomorphisms. The proof for positive equivariant symplectic homology is the same. Let $\alpha$ be the contact form on $M$ and $\alpha^\prime = f\alpha$ be another contact form, where $f: M \to \R$ is a positive function. Denote by $W$ and $W^\prime$ the fillings of $\alpha$ and $\alpha^\prime$ respectively. The key point is that given $D > \max_{x\in M}\{f(x),1/f(x)\}$ then for every $T \in \R$ there are natural maps $\psi_{\alpha\alpha^\prime}: SH_*^{S^1,+,T}(W) \to SH_*^{S^1,+,DT}(W^\prime)$ and $\psi_{\alpha^\prime\alpha}: SH_*^{S^1,+,T}(W^\prime) \to SH_*^{S^1,+,DT}(W)$ that fit into the ladder-shaped commutative diagram
\begin{equation} \label{eq:ladder}
\xymatrix{
\cdots & \cdots \\
\ar[u]
SH_*^{S^1,+,D^4T}(W) \ar[r] & SH_*^{S^1,+,D^5T}(W^\prime) \ar[ul] \ar[u]
\\
\ar[u]
SH_*^{S^1,+,D^2T}(W) \ar[r] & SH_*^{S^1,+,D^3T}(W^\prime) \ar[ul] \ar[u]
\\
\ar[u]
SH_*^{S^1,+,T}(W) \ar[r] & SH_*^{S^1,+,DT}(W^\prime) \ar[ul] \ar[u]
}
\end{equation}
where the maps in the vertical arrows are those induced by the inclusion, see \cite{Se} and \cite{HM}. Now, suppose that $SH_*^{S^1,+}(W) \simeq SH_*^{S^1,+}(W^\prime)$ is infinite-dimensional, since, otherwise, we would have $\Gamma(W) = \Gamma(W^\prime) = 0$. Then,
$$ \Gamma(W)^{-1} = \liminf_{T\to\infty} \frac{\log T}{\log r(W,T)} = \liminf_{T\to\infty} \frac{\log DT}{\log r(W,T)} \geq \liminf_{T\to\infty} \frac{\log DT}{\log r(W^\prime,DT)} = \Gamma(W^\prime)^{-1}, $$
where $r(W,T)$ is the rank of $\iota(SH_*^{S^1,+,T}(W))$. Inverting the roles of $W$ and $W^\prime$ we conclude the desired result.

\subsection{Proof of Theorem \ref{infper}}
Let $\alpha_0 = (\min f)\alpha$, $\alpha_1 = f\alpha$ and $\alpha_2 = (\max f)\alpha$. Denote by $W_0 \subset W_1 \subset W_2$ the corresponding fillings. We will follow Seidel's argument above. Given $D > \max f/\min f$ there are natural maps $\psi_{21}: SH_*^{S^1,+,T}(W_2) \to SH_*^{S^1,+,T}(W_1)$ and $\psi_{12}: SH_*^{S^1,+,T}(W_1) \to SH_*^{S^1,+,DT}(W_2)$ such that the diagram \eqref{eq:ladder} becomes
\begin{equation*}
\xymatrix{
\cdots & \cdots \\
\ar[u]
SH_*^{S^1,+,D^2T}(W_2) \ar[r] & SH_*^{S^1,+,D^2T}(W_1) \ar[ul] \ar[u]
\\
\ar[u]
SH_*^{S^1,+,DT}(W_2) \ar[r] & SH_*^{S^1,+,DT}(W_1) \ar[ul] \ar[u]
\\
\ar[u]
SH_*^{S^1,+,T}(W_2) \ar[r] & SH_*^{S^1,+,T}(W_1) \ar[ul] \ar[u]
}
\end{equation*}
Consequently, since by Theorem \ref{mainthm} $\dim \iota(SH_*^{S^1,+,T}(W_2)) = \dim SH_*^{S^1,+,T}(W_2)$, we deduce that $\dim SH_*^{S^1,+,T}(W_2) \leq \dim SH_*^{S^1,+,T}(W_1)$ and hence
\begin{align*}
 \lim_{T\to\infty}\frac{1}{T} \log P_T((\max f)\alpha) & =  \lim_{T\to\infty}\frac{1}{T} \log P^g_T((\max f)\alpha) \\
& = \lim_{T\to\infty}\frac{1}{T} \log \dim SH_*^{S^1,+,T}(W_2) \\
& \leq \liminf_{T\to\infty}\frac{1}{T} \log \dim SH_*^{S^1,+,T}(W_1) \\
& \leq \liminf_{T\to\infty}\frac{1}{T} \log P_T(f\alpha), 
\end{align*}
where the last inequality follows from \eqref{eq:2nd_page} and the fact that $f\alpha$ is non-degenerate. But, since $P_T(\alpha) = P_{cT}(c\alpha)$ for every positive constant $c$, we have that
$$ \lim_{T\to\infty}\frac{1}{T} \log P_T((\max f)\alpha) = \frac{\lim_{T\to\infty}\frac{1}{T} \log P_T(\alpha)}{\max f} = \frac{h_{top}(\alpha)}{\max f}, $$
where the last equality follows from Bowen's result in equation \eqref{bowen}. Finally, suppose that $f\alpha$ has a continuous invariant Lagrangian subbundle. Arguing as before replacing the roles of $W_2$ and $W_1$ by $W_1$ and $W_0$ respectively, we conclude that
\begin{align*}
\liminf_{T\to\infty}\frac{1}{T} \log P_T(f\alpha) & = \liminf_{T\to\infty}\frac{1}{T} \log P^g_T(f\alpha) \\
& = \liminf_{T\to\infty}\frac{1}{T} \log \dim SH_*^{S^1,+,T}(W_1) \\
& \leq \lim_{T\to\infty}\frac{1}{T} \log \dim SH_*^{S^1,+,T}(W_0) \\
& = \lim_{T\to\infty}\frac{1}{T} \log P^g_T((\min f)\alpha) \\
& = \lim_{T\to\infty}\frac{1}{T} \log P_T((\min f)\alpha) \\
& = \frac{h_{top}(\alpha)}{\min f}, 
\end{align*}
where the second equality follow from Theorem \ref{mainthm} and the fact that $f\alpha$ has a continuous invariant Lagrangian subbundle.
\qed

\medskip

{\it Acknowledgement.} We are grateful to Ivan Smith for pointing out to us Remark \ref{rem:ivan}.

\end{document}